\documentclass{amsart}
\pdfoutput=1 
\usepackage{verbatim}
\usepackage{amssymb, mathtools, amsthm}
\usepackage[shortlabels]{enumitem}
\usepackage{tikz}
\usepackage{mathabx} 
\usetikzlibrary{calc,positioning}
\usepackage{tikz-3dplot}
\usetikzlibrary{3d}
\usepackage[labelfont=up]{subcaption}
\usepackage{graphicx}
\usepackage{hyperref}

\hypersetup{
    pdftitle={RAAGedy right-angled Coxeter groups II: in the
      quasiisometry class of the tree RAAGs},    
    pdfkeywords={Right-angled Coxeter group, right-angled Artin group, RACG,
  RAAG, quasiisometry}, 
    colorlinks=true,       
    linkcolor=black,          
    citecolor=black,        
    filecolor=black,      
    urlcolor=black           
}

\usepackage{fancyhdr}
\fancypagestyle{tp}{
\fancyhead{}
\fancyhead[C]{Author's accepted manuscript. The published version,
  subject to  editorial changes, will appear in Proceedings of the
  American Mathematical Society.
}
}

\NewCommandCopy{\proofqedsymbol}{\qedsymbol}
\newcommand{\exampleqedsymbol}{$\diamond$}

\AtBeginEnvironment{proof}{\renewcommand{\qedsymbol}{\proofqedsymbol}}

\theoremstyle{definition}

\theoremstyle{plain}
\newtheorem{theorem}{Theorem}[section]
\newtheorem{lemma}{Lemma}[section]

\theoremstyle{remark}
\newtheorem{claim}{Claim}[theorem]

\theoremstyle{definition}
\newtheorem{definition}{Definition}[section]
\newtheorem{example}{Example}[section]

\AtBeginEnvironment{example}{%
  \pushQED{\qed}\renewcommand{\qedsymbol}{\exampleqedsymbol}%
  }
\AtEndEnvironment{example}{\popQED\endexample}

\def\makeautorefname#1#2{\expandafter\def\csname#1autorefname\endcsname{#2}}
\let\fullref\autoref

\makeautorefname{theorem}{Theorem} 
\makeautorefname{lemma}{Lemma} 
\makeautorefname{proposition}{Proposition} 
\makeautorefname{corollary}{Corollary}
\makeautorefname{numberedremark}{Remark} 
\makeautorefname{definition}{Definition}
\makeautorefname{example}{Example}
\makeautorefname{section}{Section}
\makeautorefname{subsection}{Section}
\makeautorefname{subsubsection}{Section}
\makeautorefname{claim}{Claim}
\makeautorefname{question}{Question}

\makeatletter 
\let\c@lemma=\c@theorem 
\makeatother
\makeatletter 
\let\c@proposition=\c@theorem 
\makeatother
\makeatletter 
\let\c@corollary=\c@theorem 
\makeatother
\makeatletter 
\let\c@definition=\c@theorem 
\makeatother
\makeatletter 
\let\c@example=\c@theorem 
\makeatother
\makeatletter 
\let\c@question=\c@theorem 
\makeatother
\makeatletter 
\let\c@numberedremark=\c@theorem 
\makeatother
\makeatletter
\@addtoreset{claim}{proposition}
\makeatother
\makeatletter
\@addtoreset{claim}{lemma}
\makeatother
\newenvironment{claimproof}[1][\unskip]{\vspace{1ex}\noindent{\it
Proof of Claim #1:}\hspace{0.5em}}{\hfill$\lozenge$\vspace{1ex}}
\mathtoolsset{centercolon} 

\newcommand{\from}{\colon\thinspace} 
\DeclareMathOperator{\lk}{lk}
\DeclareMathOperator{\st}{st}
\DeclareMathOperator{\supp}{supp}
\DeclareMathOperator{\hull}{hull}
\newcommand{\diag}{\boxslash}

\newcommand{\join}{*}
\newcommand{\edge}{\mathbin{\tikz[baseline=-\the\dimexpr\fontdimen22\textfont2\relax
    ]{\filldraw (0,0) circle (.5pt) (.2,0) circle (.5pt);\draw
      (0,0)--(.2,0);}}}
\newcommand{\longdash}{\mathbin{\tikz[baseline=-\the\dimexpr\fontdimen22\textfont2\relax ]{\draw (0,0)--(1,0);}}}
\renewcommand{\setminus}{-}
\newcommand{\odd}{\mathbb{O}}
\newcommand{\even}{\mathbb{E}}
\newcommand{\free}{\mathbb{F}}

\begin{document}
\title[RACGs QI to tree RAAGs]{RAAGedy right-angled Coxeter groups II:
  in the quasiisometry class of the tree RAAGs}
\author{Christopher H.\ Cashen}
\address{Faculty of Mathematics\\University of  Vienna\\Oskar-Morgenstern-Platz 1\\1090 Vienna, Austria\\\href{https://orcid.org/0000-0002-6340-469X}{\includegraphics[scale=.75]{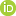}  0000-0002-6340-469X}}
\email{christopher.cashen@univie.ac.at}
\date{June 10, 2025}
\thanks{This research was funded  by the Austrian Science Fund (FWF)
  \href{https://doi.org/10.55776/P34214}{10.55776/P34214}.
} 

\begin{abstract}
We classify two-dimensional right-angled
Coxeter groups that are quasiisometric to a
right-angled Artin group defined by a tree, and show that
when this is true the right-angled Coxeter group actually contains a
visible finite index right-angled Artin subgroup.

\end{abstract}

\maketitle
\thispagestyle{tp}

\section{Introduction}\label{sec:intro}
A \emph{tree RAAG} is a right-angled Artin group whose presentation
graph is a tree of diameter at least three.
Such a group is the fundamental group of a compact 3--manifold
with boundary whose JSJ decomposition, in the 3--manifold sense,
contains only Seifert fibered pieces.
A corollary of Behrstock and Neumann's \cite{BehNeu08} quasiisometry
classification of graph manifolds is that all tree RAAGs belong to a
single quasiisometry class.
The diameter condition excludes the
well-understood quasiisometry classes $\mathbb{Z}$, $\mathbb{Z}^2$, and
$\free_2\times\mathbb{Z}$. 

Nguyen and Tran \cite{NguTra19} classified planar, triangle-free
graphs $\Gamma$ such that the right-angled Coxeter group (RACG)
$W_\Gamma$ with presentation graph $\Gamma$ is in the quasiisometry
class of the tree RAAGs.
Dani and Levcovitz \cite{DanLev24} subsequently showed that when
Nguyen and Tran's conditions are satisfied then $W_\Gamma$ actually 
has a finite index tree RAAG subgroup that is \emph{visible} in
$\Gamma$, in a sense that will be made precise later.

Nguyen and Tran's proof uses planarity in an essential way twice: once
to invoke the Jordan Curve Theorem to construct a certain
suspension-decomposition of $\Gamma$, and again to say that $W_\Gamma$
is virtually a 3--manifold group.
Then their suspension-decomposition corresponds to the JSJ
decomposition of the manifold and they apply Behrstock and
Neumann's result.
Dani and Levcovitz's argument also explicitly uses a planar embedding
of $\Gamma$. 
However, there are easy examples, such as the graph in \fullref{fig:nonplanar_raag_tree}, of nonplanar graphs $\Gamma$ such that
$W_\Gamma$ contains a finite index tree RAAG subgroup (see \fullref{ex:nonplanar_raag_tree_redux}), so planarity is not a
necessary condition.  
 \begin{figure}[h]
    \centering
\begin{tikzpicture}
    \coordinate (0) at (0,0);
    \coordinate (1) at (0,.5);
    \coordinate (2) at (0,1);
    \coordinate (3) at (-1,1);
    \coordinate (4) at (-1,0);
    \coordinate (5) at (1,.5);
    \coordinate (6) at (-.5,.75);
    \coordinate (7) at (-.5,.25);
    \coordinate (8) at (.5,.5);

    \filldraw (0) circle (1pt) (1) circle (1pt) (2) circle (1pt) (3)
    circle (1pt) (4) circle (1pt) (5) circle (1pt) (6) circle (1pt)
    (7) circle (1pt) (8) circle (1pt);
    \draw (0)--(1)--(2)--(4)--(7)--(1)--(8)--(5)--(0)--(3)--(2)
    (3)--(6)--(1) (0)--(4) (2)--(5);
  \end{tikzpicture}
  \caption{A nonplanar, triangle-free graph $\Gamma$ such that $W_\Gamma$ is
      in the quasiisometry class of the tree RAAGs.}
    \label{fig:nonplanar_raag_tree}
  \end{figure}

  We give a complete classification of triangle-free graphs $\Gamma$ such that $W_\Gamma$ is in the
  quasiisometry class of the tree RAAGs, and show that in this case $W_\Gamma$
  contains a visible finite index tree RAAG subgroup.
See \fullref{maintheorem}.
The proof uses the group theoretic version of JSJ
decompositions (over 2--ended subgroups).
Such decompositions of RAAGs and RACGs can be read
  off from their presentation graphs.
One step in the proof is to describe what the JSJ decomposition of a
group in the quasiisometry class of the tree RAAGs must look like. 
The second is to show that, for a 2--dimensional RACG, these
conditions are already strong enough to produce a visible finite index RAAG subgroup.

There are infinitely many commensurability classes of tree
RAAGs \cite{CasKazZak19,CasKazZak21}, so the main theorem does not
say
that quasiisometry to a particular tree RAAG implies commensurability
to that same tree RAAG.
Nevertheless, we find the commensurability aspect of the theorem surprising. 
Usually when a geometric relation can be promoted to an algebraic
relation it is because there is some strong rigidity
phenomenon at play, but the class of tree RAAGs is quite flexible. 
For example, the results of Huang and Kleiner \cite{MR3761106} on
groups quasiisometric to RAAGs with finite outer automorphism group do
not apply to tree RAAGs.
\nocite{CasDanEdl} 
\subsection*{Acknowledgements}
This paper is a spin-off of \cite{CasDanEdl}, where the question is when $W_\Gamma$ is quasiisometric to any RAAG whatsoever.

Alexandra Edletzberger had already noticed that Nguyen and Tran's decomposition was related to the JSJ decomposition, and that cycles of cuts should be obstructions to being quasiisometric to a RAAG \cite[Remark~1.38, Remark~4.29]{Edl24thesis}.

\section{Preliminaries}\label{sec:prelim}

\subsection{Terminology}
The \emph{right-angled Artin group
with presentation graph $\Delta$}, for a finite simplicial graph
$\Delta$, is the group $A_\Delta$ presented by a generator for each
vertex of $\Delta$ and a
commuting relation between two generators when they are connected by
an edge of $\Delta$.
The \emph{right-angled Coxeter group with presentation graph $\Gamma$},
$W_\Gamma$, is defined similarly, with additional relations saying that
each generator has order 2.
In both cases these groups have a geometric action on a CAT(0) cube
complex, which has dimension at most 2 when the presentation graph is triangle-free.

$\free_n$ is the free group of rank $n$, and
$\free$ is $\free_n$ for some unspecified $n\geq 2$ when the precise rank is
not important.
A \emph{basis} of $\free$ means a free generating set.

An \emph{induced} subgraph of a simplicial graph is a subgraph that
contains all possible edges between its vertices.
If $\Upsilon'$ is an induced subgraph of $\Upsilon$ then the inclusion
at the graph level extends to inclusions $A_{\Upsilon'}< A_\Upsilon$
and $W_{\Upsilon'}<W_\Upsilon$.
A subgroup of a RAAG or RACG defined in this way is called a \emph{special subgroup}.

If $\Gamma$ is a presentation graph of a RACG then a \emph{thick join}
means a join subgraph $A\join B$ of $\Gamma$ such that $W_A$ and $W_B$
are infinite groups, which happens when  $A$ and $B$ are incomplete
subgraphs.
A \emph{suspension} is a join in which $A$ is a 2--anticlique.
If $A$ is a 2--anticlique and $|B|>2$ then $A$ is called the
\emph{pole} or the \emph{suspension points} of $A\join B$ and $B$ is
called the \emph{suspended points}.
A vertex is \emph{essential} if it has valence at least 3.
A \emph{cone vertex} is everybody's neighbor.
If $v$ is a vertex, its \emph{link} $\lk(v)$ is its set of neighbors
and its \emph{star} is $\st(v):=\lk(v)\cup\{v\}$.
A graph is \emph{biconnected} if it is connected with no cut
vertex.
Thus, a single edge is biconnected. 

Two subsets of a metric space are \emph{$r$--coarsely equivalent} if
their Hausdorff distance is at most $r$, and \emph{coarsely
  equivalent} if their Hausdorff distance is finite.
Coarse equivalence gives an equivalence relation that is preserved by
quasiisometries.

\subsection{Visible RAAG subgroups}\label{danilevcovitz}
If $\Gamma$ is a graph, its complement graph $\Gamma^c$ is the graph
on the same vertex set that has an edge precisely when $\Gamma$ does
not. 
Suppose $\Lambda\subset\Gamma^c$ is a disjoint union of
trees $\Lambda_0$ and $\Lambda_1$.
For a set of vertices $V$ in $\Gamma$, if $V$ is contained in one of
the $\Lambda_i$ then $\hull_\Lambda(V)$ is the set of vertices of
$\Gamma$ contained in the convex hull of $V$ in $\Lambda_i$.
  The \emph{commuting graph} $\Delta$ of $\Lambda$ is the graph that
  has a vertex for each edge of $\Lambda$, with an edge between two of
  them if their support spans a square in $\Gamma$.
    There is a natural homomorphism $A_\Delta\to W_\Gamma$
    defined by 
    sending the generator of $A_\Delta$ corresponding to vertex $\{a,b\}$ of
    $\Delta$ to the element $ab\in W_\Gamma$.
it is easy to come up with examples where this homomorphism is not
injective, but Dani and Levcovitz \cite{DanLev24} give conditions that
imply the homomorphism is injective and the image has finite index.
    
\begin{theorem}[{\cite[Theorem~4.8, Corollary~3.33, Remark~4.3]{DanLev24}}]\label{thm:DL}
  Let $\Gamma$ be an incomplete, triangle-free graph without
  separating cliques.
  Suppose $\Lambda$ is a disjoint union of trees $\Lambda_0$ and
  $\Lambda_1$ in  $\Gamma^c$ such that $\Lambda$ spans $\Gamma$ and
  for each $i$, no two vertices of $\Lambda_i$ are adjacent in
  $\Gamma$.
  Let $\Delta$ be the commuting graph of $\Lambda$.
  
If  $\Lambda$ satisfies the following conditions, then the
homomorphism $A_\Delta\to W_\Gamma$ is injective and its image has
index 4 in $W_\Gamma$.
  \begin{enumerate}
  \item[$\mathcal{R}_3$:] If $\{a,b\}\join\{c,d\}$ is a square in
    $\Gamma$ then
    $\hull_\Lambda\{a,b\}\join\hull_\Lambda\{c,d\}\subset\Gamma$.
    \item[$\mathcal{R}_4$:] If $a\edge b$ is an edge in a cycle
      $\gamma\subset\Gamma$ then there is a square $\{a,a'\}\join\{b,b'\}$ with $a',b'\in\hull_\Lambda(\gamma)$.
    \end{enumerate}
\end{theorem}
Call a graph $\Lambda$ satisfying the hypotheses of \fullref{thm:DL}  a
\emph{finite index Dani-Levcovitz $\Lambda$} (FIDL--$\Lambda$).
When we have a FIDL--$\Lambda$, the commuting graph $\Delta$ gives a
`visible' finite index RAAG subgroup of $W_\Gamma$.
It is `visible' in the sense that vertices of $\Delta$ correspond to
diagonals of squares of $\Gamma$, and vertices of $\Delta$ span an
edge precisely when their support spans a square in $\Gamma$.

\subsection{CFS graphs}
For $W_\Gamma$ to be quasiisometric to a
RAAG it must have at most quadratic divergence, which is
characterized by $\Gamma$ having property \emph{CFS} \cite{DanTho15}.
\begin{definition}[{cf.\cite[Definition~2.3]{CasEdl}}] 
  The \emph{diagonal graph} $\diag(\Gamma)$ of $\Gamma$ is the graph
  with an edge for each induced square of $\Gamma$, with vertices
  representing diagonal pairs.
  The \emph{support} of a vertex of $\diag(\Gamma)$ is the two
  vertices of $\Gamma$ in the corresponding diagonal pair. The support
  of a subset of $\diag(\Gamma)$
  is the union of supports of its vertices. 
$\Gamma$ is \emph{CFS} if $\diag(\Gamma)$ contains a component whose
support is all non-cone vertices of $\Gamma$.
\end{definition}

We will not invoke the CFS property directly, but $\diag(\Gamma)$ will
figure prominently.
Consider that $A_\Delta$ is 1--ended when $\Delta$ is connected and
not a single vertex, and $W_\Gamma$ is 1--ended when $\Gamma$ is incomplete
without separating cliques.
Thus, in the 1--ended case,  for a FIDL--$\Lambda$ as in the
previous section, we may identify $\Delta$ with a subgraph of
$\diag(\Gamma)$, since $\Delta$ is connected and an edge in $\Delta$ corresponds to a square in $\Gamma$.

\subsection{JSJ decompositions}\label{prelim:jsj}
We assume familiarity with Bass-Serre theory.
In this paper we are interested in JSJ decompositions of finitely
presented groups over 2--ended subgroups (equivalently, over virtually
$\mathbb{Z}$ subgroups), which are 
graph of groups decompositions that encode all possible splittings of
the group over 2--ended subgroups.
The existence of non-trivial JSJ decompositions, for finitely
presented groups not commensurable to a surface group, is a quasiisometry
invariant \cite{Pap05}, but the precise structure of a particular JSJ
decomposition is not.
One can make a canonical object by passing to the Bass-Serre tree of
any JSJ decomposition and
collapsing cylinders to make the \emph{JSJ tree of cylinders}
\cite{MR3758992}. 
It is a bipartite tree where one part consists of \emph{cylinder}
vertices, which correspond to commensurability classes of universally
elliptic 2--ended splitting subgroups, and one part is made up of
\emph{rigid} and \emph{hanging} vertices that belong to more than one
cylinder.
Hanging vertices contain the non-(universally elliptic) 2--ended
splitting subgroups; that is, they contain collections of 2--ended
splitting subgroups that give incompatible group splittings.
Rigid vertices correspond to subgroups that are not split by any
2--ended splitting of the ambient group.

It follows from Papasoglu's work \cite{Pap05} (cf.\ \cite{CasMar17})
that a quasiisometry $\phi\from G\to G'$ between
two finitely presented groups induces an isomorphism of their JSJ trees of
cylinders, $\phi_*\from T\to T'$, that preserves vertex type:
cylinder/rigid/hanging. 
Furthermore, if $X$ is a tree of spaces for $G$ over $T$, and
similarly $X'$ for $G'$ over $T'$, then restriction of $\phi$ to each vertex space
$X_v$ is uniformly close to a quasiisometry $\phi_v\from
X_v\to X'_{\phi_*(v)}$ that induces a bijection of incident edge
spaces by taking $X_e\subset X_v$ to a set uniformly coarsely equivalent to
$X'_{\phi_*(e)}\subset X'_{\phi_*(v)}$.

The quotient of the action of a group on its JSJ tree of cylinders
gives the \emph{graph of cylinders}.
Since the JSJ tree of cylinders can be recovered by development, this
graph of groups contains information on quasiisometry invariants for
the group, such as the vertex types that occur, the quasiisometry
classes the vertex groups, and the relative quasiisometry types of the
vertices relative to the peripheral pattern of their incident edge groups, which will be
discussed in \fullref{sec:basicpatterns}.

Work of  Clay \cite{Cla14} and Margolis \cite{Mar20}, says that
the graph of cylinders of a RAAG can be described `visually' in
terms of the presentation graph $\Delta$ as follows, where in each
case the local group is the special subgroup defined by the relevant subgraph:
\begin{itemize}
  \item Cylinders are stars of cut vertices of $\Delta$.
\item Rigid vertices are maximal biconnected subgraphs of $\Delta$
  that are not contained in a single cylinder.
\item Hanging vertices do not occur.
  \item There is an edge between a cylinder and rigid vertex when they intersect.
  \end{itemize}

Edletzberger \cite{Edl24}, building on work of Mihalik and Tschantz
\cite{MihTsc09} and Dani and Thomas \cite{DanTho17}, gave a visual description of the graph
of cylinders of a 1--ended RACG with a triangle-free presentation
graph $\Gamma$ that is not a cycle graph.
The 2--ended subgroups of interest manifest as special subgroups
generated by \emph{cuts}
     $\{a-b\}$, meaning $a$ and $b$ are not adjacent and either:
     \begin{itemize}
     \item $\{a,b\}$ is a cut pair, that is, $\Gamma\setminus\{a,b\}$
       is not connected.
       \item $\{a,b\}$ is not a cut pair, but there is a vertex
         $c\in\lk(a)\cap\lk(b)$ such that the 2--path $a\edge c\edge
         b$ disconnects $\Gamma$. In this case $a\edge c\edge b$ is a
         \emph{cut 2--path}.
                \end{itemize}
The distinction between 2--ended splitting
subgroups that are non-(universally elliptic) vs universally elliptic
is reflected topologically: A cut gives a non-(universally elliptic)
splitting if there is some other cut that \emph{crosses} it, in the
sense that each cut contains vertices in different components of the
complement of the other. A cut gives a universally elliptic splitting
if it is \emph{uncrossed}, ie, not crossed by any cut. 
\begin{itemize}
   \item Cylinders are suspensions
     $\{a,b\}\join(\lk(a)\cap\lk(b))$ for uncrossed
     cut $\{a-b\}$. 
         \item Rigid vertices are sets of essential vertices of
    $\Gamma$ of size at least four that cannot be separated by a cut and that are maximal with
    respect to inclusion among such sets.
\item Hanging vertices consist of maximal collections of pairwise crossing
  cut pairs or pairwise crossing cut 
  2--paths. (Cut pairs cannot cross cut 2--paths, and vice versa
 \cite[after Definition~2.8]{CasEdl}.)
 \item There is an edge between a cylinder  $\{a,b\}\join(\lk(a)\cap\lk(b))$ and a rigid/hanging
   vertex if the latter contains a cut $\{a-b\}$. In this case the edge
   corresponds to intersection of subgraphs. 
\end{itemize}

\subsection{Peripheral patterns}\label{sec:basicpatterns}
A \emph{peripheral pattern} $\mathcal{P}(G,\mathcal{H})$ in a group $G$ relative to a collection
$\mathcal{H}:=\{H_0,\dots,H_{n-1}\}$ of subgroups of $G$ is the set of coarse
equivalence classes of left cosets of the $H_i$ in $G$.
When $\mathcal{H}$ consists of 2--ended subgroups the peripheral
pattern it generates may be called a \emph{line
  pattern}.
A map between two groups with peripheral patterns is 
\emph{pattern preserving} if it takes each coset of one pattern to within
uniformly bounded Hausdorff distance of one in the other, and vice
versa, inducing a bijection between patterns.
The \emph{relative quasiisometry type} of $\mathcal{P}(G,\mathcal{H})$
is the equivalence class of spaces with peripheral patterns up to
pattern preserving quasiisometry.

Line patterns in free groups have been considered before 
\cite{Ota92,CasMac11,Cas16}.
Associated to a line pattern there is a \emph{decomposition space},
which is the quotient of the boundary of the free group obtained by
identifying the two endpoints of each line in the pattern.
Pattern preserving quasiisometries induce homeomorphisms of
decomposition spaces.
Two cases are of particular interest for us.
It turns out that each of these types defines a single relative
quasiisometry type of line pattern.

\begin{definition}
  A \emph{surface type (line) pattern} in $\free$ is a peripheral pattern 
  $\mathcal{P}(\free,\mathcal{H})$ with $\mathcal{H}=\{\langle f_0\rangle,\dots,\langle
f_{n-1}\rangle\}$ such that, up to passing to roots, conjugates, and
inverses, and removing duplicate entries, there is a compact surface
with fundamental group $\free$ such that $\{f_0,\dots,f_{n-1}\}$ are
the elements of $\free$ represented by the boundary curves.
Equivalently, its decomposition space is a circle. 
\end{definition}

\begin{definition}
  A \emph{basic (line) pattern} in $\free$ is a peripheral pattern 
  $\mathcal{P}(\free,\mathcal{H})$ with $\mathcal{H}=\{\langle f_0\rangle,\dots,\langle
f_{n-1}\rangle\}$ such that, up to passing to roots, conjugates, and
inverses, and removing duplicate entries, $\{f_0,\dots,f_{n-1}\}$ is a subset of a basis
of $\free$.
Equivalently, its decomposition space is totally disconnected. 
\end{definition}

In the present work we are interest not in line patterns in free
groups, but in \emph{plane patterns} in $\free\times\mathbb{Z}$.
Let $z$ be a generator of the center.
By a \emph{plane pattern} in $\free\times\mathbb{Z}$, we mean a peripheral
pattern $\mathcal{P}(\free\times\mathbb{Z},\mathcal{H})$ where
$\mathcal{H}$ consists of $\mathbb{Z}^2$ subgroups, which we might as well assume, up to
coarse equivalence,  are maximal $\mathbb{Z}^2$
subgroups.
A maximal $\mathbb{Z}^2$ subgroup of $\free\times\mathbb{Z}$ is of the
form $\langle f\rangle\times\langle z\rangle$, where $f\in \free$ and
$z\in\mathbb{Z}$ are
nontrivial elements that are not proper powers.
Using the results and terminology of \cite{Pap05}, the only separating quasilines of
$\free\times\mathbb{Z}$ are coarsely equivalent to the central
$\mathbb{Z}$.
Thus, a quasiisometry $\phi\from \free\times\mathbb{Z}\to
\free'\times\mathbb{Z}$ takes each coset of the central $\mathbb{Z}$ to
within bounded Hausdorff distance, depending only on the quasiisometry
constants, of a coset of the central $\mathbb{Z}$.
From this it follows that $\pi_{\free'}\circ \phi|_{\free\times\{1\}}\from \free\to
\free'$, where $\pi_{\free'}\from \free'\times\mathbb{Z}$ is projection to the
first factor given by killing the center, is a quasiisometry.
Conversely, $(\pi_{\free'}\circ \phi|_{\free\times\{1\}})\times
\mathrm{Id}_\mathbb{Z}\from \free\times\mathbb{Z}\to
\free'\times\mathbb{Z}$ is a quasiisometry (but not necessarily one
that is close
to $\phi$).
The projections $\pi_\free\from \free\times\mathbb{Z}\to \free$ and $\pi_{\free'}\from
\free'\times\mathbb{Z}\to \free'$ send peripheral plane patterns to peripheral
line patterns, and 
the original quasiisometry $\phi$ preserves plane patterns if and only if 
$\pi_{\free'}\circ \phi|_{\free\times\{1\}}$ preserves the projected line patterns.
We extend the basic/surface terminology to the
plane patterns according to the type of the projection.
More generally, extend the terminology to patterns of 
virtually $\mathbb{Z}^2$ subgroups
in  virtually $\free\times\mathbb{Z}$ groups, according to the type of the
pattern obtained by pushing forward the pattern by the map induced by
the restriction to a finite index $\free\times\mathbb{Z}$
subgroup.

\begin{lemma}\label{raags_are_basic}
  Let $\Delta$ be a finite tree of diameter at least three, so that 
$A_\Delta$ is a 1--ended RAAG with a nontrivial graph of cylinders.
For every cylinder in the graph of cylinders, the peripheral pattern coming from the
incident edge groups is basic.
\end{lemma}
\begin{proof}
  Since $\Delta$ is a tree, every non-leaf is a cut vertex, so 
  cylinders are stars of non-leaf vertices $v$ of $\Delta$.
If $v$ is a non-leaf then $\lk(v)=\{w_0,\dots,w_{n-1},\ell_0,\dots,\ell_{m-1}\}$, where
the $w_i$ are non-leaves and the $\ell_j$ are leaves.
Since the diameter of $\Delta$ is greater
than 2,  $n>0$ and
$m+n\geq 2$.
The cylinder group is $A_{\st(v)}=\langle v\rangle\times\langle
w_0,\dots,w_{n-1},\ell_0,\dots,\ell_{m-1}\rangle\cong \mathbb{Z}\times
\free_{m+n}$.
Since $\Delta$ is a tree, the maximal biconnected subgraphs are single
edges, and those that are not contained in a single cylinder are the
non-leaf edges. 
Thus, the neighbors of the cylinder $\st(v)$ in
the graph of cylinders are the rigid vertices $v\edge w_i$ for $0\leq
i<n$, with $A_{\{v,w_i\}}\cong\mathbb{Z}^2$.
These give a basic plane pattern, since $\{w_0,\dots,w_{n-1}\}$ is a subset
of the basis $\{w_0,\dots,w_{n-1},\ell_0,\dots,\ell_{m-1}\}$ of  the $\free_{m+n}$
factor of $A_{\st(v)}$.
\end{proof}

\begin{example}\label{ex:nonbasicjsj}
  Let $\Gamma$ be the nonplanar graph in \fullref{fig:nonbasic}.
  Let $W_\Gamma$ be the RACG whose presentation graph is $\Gamma$, so
  that the generators of $W_\Gamma$ are $s_i$ where $s_i$ corresponds
  to vertex $i$ of $\Gamma$.
  (For compactness of notation, we omit the `$s$' in figures and subscripts.)
The graph of cylinders of $W_\Gamma$ is shown in
\fullref{fig:jsjgoc_nonbasic}, where the labels are the
cylinder groups and each unlabelled vertex is a rigid vertex whose group is
the intersection of its neighboring cylinder groups, which in each
case is isomorphic to the product of two infinite dihedral groups.
All edge groups are isomorphic to their neighboring rigid vertex
group, and all edge maps are inclusion.
\begin{figure}[h]
  \begin{subfigure}{.35\textwidth}
      \centering
      \tdplotsetmaincoords{100}{50} 
  \begin{tikzpicture}[tdplot_main_coords]\tiny
  \pgfmathsetmacro{\phi}{sqrt(3)/2}
  \coordinate[label={[label distance=0pt] 0:$0$}]  (0) at (1,0,0);
  \coordinate[label={[label distance=0pt] -45:$1$}]  (1) at (-.5,\phi,0);
  \coordinate[label={[label distance=0pt] 180:$2$}]  (2) at (-.5,-\phi,0);
  \coordinate[label={[label distance=0pt] 0:$3$}]  (3) at (0,0,1);
  \coordinate[label={[label distance=0pt] 0:$4$}]  (4) at (0,0,-1);
  \coordinate[label={[label distance=0pt] -45:$5$}]  (5) at ($(0)!0.5!(1)$);
  \coordinate[label={[label distance=0pt] -90:$6$}]  (6) at ($(1)!0.5!(2)$);
  \coordinate[label={[label distance=0pt] 90:$7$}]  (7) at
  ($(2)!0.5!(0)$);
  \coordinate[label={[label distance=-1pt] 0:$8$}]  (8) at (0,0,0);
    \draw (1)--(6)--(2)--(7)--(0);
  \draw (3)--(0)--(4) (3)--(1)--(4) (3)--(2)--(4);
  \draw (0)--(1);
     \filldraw (0) circle (1pt) (1) circle (1pt) (2) circle (1pt) (3)
  circle (1pt) (4) circle (1pt) (5) circle (1pt) (6) circle (1pt) (7)
  circle (1pt) (0,0,0) circle (1pt);
  \draw (3)--(8)--(4);
\end{tikzpicture}
  \subcaption{$\Gamma$}
  \label{fig:nonbasic}
\end{subfigure}
\hfill
  \begin{subfigure}{.6\textwidth}
\centering
  \begin{tikzpicture}\tiny
 \coordinate[label={[label distance=0pt] 180:$W_{\{3,4\}}\times
   W_{\{0,1,2,8\}}$}]  (0) at (0,0);
 \coordinate[label={[label distance=0pt] 0:$W_{\{0,1\}}\times W_{\{3,4,5\}}$}]  (1) at (0:1);
\coordinate[label={[label distance=0pt] 0:$W_{\{1,2\}}\times W_{\{3,4,6\}}$}]  (2) at (-120:1);
  \coordinate[label={[label distance=0pt] 0:$W_{\{0,2\}}\times
    W_{\{3,4,7\}}$}]  (3) at (120:1);
  \draw (0)--(1) (0)--(2) (0)--(3);
  \filldraw (0) circle (1pt) (1) circle (1pt) (2) circle (1pt) (3)
  circle (1pt) ($(0)!.5!(1)$) circle (1pt) ($(0)!.5!(2)$) circle (1pt) ($(0)!.5!(3)$) circle (1pt) ;
\end{tikzpicture}
  \subcaption{Graph of cylinders of $W_\Gamma$.}
  \label{fig:jsjgoc_nonbasic}
\end{subfigure}
\caption{Graph $\Gamma$ and graph of cylinders of $W_\Gamma$ in \fullref{ex:nonbasicjsj}.}
\end{figure}

The middle cylinder vertex $\mathcal{C}:=W_{\{3,4\}}\times W_{\{0,1,2,8\}}$ has
peripheral pattern generated by $\mathcal{H}:=\{W_{\{0,1\}}\times
W_{\{3,4\}},\, W_{\{1,2\}}\times W_{\{3,4\}},\, W_{\{0,2\}}\times
W_{\{3,4\}}\}$.
The group $\mathcal{C}$ has an index 4 subgroup isomorphic to 
$\free_3\times\mathbb{Z}$, namely $\langle w,x,y\rangle\times\langle
z\rangle$, where $w:=s_0s_8$, $x:=s_0s_1$, $y:=s_1s_2$, and $z:=s_3s_4$.
Restriction to the finite index subgroup induces a pattern preserving
quasiisometry between $\mathcal{P}(\mathcal{C},\mathcal{H})$  and $\free_3\times \mathbb{Z}$ with plane
pattern generated by $\{\langle x\rangle\times\langle z\rangle,\,\langle
y\rangle\times\langle z\rangle,\,\langle xy\rangle\times\langle
z\rangle\}$.
Projection to the $\free_3$ factor gives line pattern generated by
$\{x,y,xy\}$.
Since the generators live in a free factor $\langle x,y\rangle$ of
$\free_3$, it follows that the decomposition space for $\mathcal{P}(\langle
x,y\rangle,\{x,y,xy\})$ embeds into the decomposition space for $\mathcal{P}(\free_3,\{x,y,xy\})$.
The former is a surface type pattern:
it matches the standard description of the fundamental group and
boundary curves of a 3--holed sphere.
Therefore, the decomposition space for $\mathcal{P}(\free_3,\{x,y,xy\})$ contains
circles; in particular, it is not totally disconnected, so the line
pattern is not basic. 
Thus, $W_\Gamma$ is not quasiisometric to a tree RAAG, since the non-basic
plane pattern in the cylinder  $W_{\{0,1,2,8\}}\times W_{\{3,4\}}$ of
the RACG  cannot match the plane patterns in cylinder vertices
of tree RAAGs, which by \fullref{raags_are_basic} are all
basic. 
\end{example}

In the planar case, according to Nguyen and Tran's argument, $\Gamma$
decomposes into a tree of maximal suspensions, with neighbors
determined by two suspensions sharing a square.
Playing the game as in the example by restricting to a finite index
torsion-free subgroup and then projecting to the free factor, one can see that there are
only two possibilities for a maximal suspension 
$\{a,b\}\join\{c_0,\dots,c_{n-1}\}$:
either there are fewer than $n$ neighbors and the line pattern is
basic, or there are exactly $n$ neighbors and the line pattern is
surface type.
Nguyen and Tran use `suspension of $n$ vertices has exactly $n$
neighboring suspensions' as an obstruction to being quasiisometric to
a tree RAAG. 
\fullref{ex:nonbasicjsj} illustrates that in the nonplanar case there
are more varied ways for a pattern to fail to be
basic; $\mathcal{C}$ has a non-basic pattern, but it
is a suspension of 4 vertices with only 3 neighboring
suspensions.

\section{RACGs quasiisometric to tree RAAGs}\label{sec:raagtrees}

\begin{theorem}\label{maintheorem}
  Let $\Gamma$ be an incomplete, triangle-free graph without
  separating cliques. The following are equivalent:
  \begin{enumerate}
              \item $W_\Gamma$ has an index four visible tree RAAG subgroup.\label{item:commraagtree}
      \item $W_\Gamma$ is in the quasiisometry class of the tree RAAGs.\label{item:qiraagtree}
    \item $W_\Gamma$ has a nontrivial graph of cylinders with no
      hanging vertices and:
      \begin{enumerate}
      \item Cylinders are virtually $\free\times\mathbb{Z}$.\label{item:cyclinderfreez}
        \item Rigid vertices are virtually $\mathbb{Z}^2$. Each is
          adjacent to two cylinders, with edge group equal to the
          rigid vertex group. \label{item:rigidzsquared}
        \item In each
          cylinder the incident edge groups form a basic plane pattern.\label{item:basicplanepattern}
      \end{enumerate}\label{item:jsjgoc}
    \item $\Gamma$ has the following structure:
      \begin{enumerate}
      \item Maximal thick joins are non-square
        suspensions.
        Every vertex and edge is contained in a maximal
        suspension.
        The pole of each maximal suspension gives a cut, in the sense
        that it is either a cut pair or
        forms a cut 2--path with one of the suspended vertices.
        All cuts come from the pole of a maximal suspension in one of
        these two ways.\label{item:cylinders}
      \item Every set of at least four essential vertices that is not
        separated by a cut is a square, each of
        whose diagonals is the pole of a
        maximal suspension. \label{item:rigid}
        \item There does not exist a maximal suspension
          $\{a,b\}\join\{c_0,\dots,c_{n-1}\}$ and $3\leq m\leq n$ such
          that for all $0\leq i<m$ the pair $\{c_i,
          c_{(i+1)\mod m}\}$ is the pole of a
          maximal suspension. \label{item:nosuspensioncycle}
        \end{enumerate}\label{item:raagtreegraphconditions}
  \end{enumerate}
\end{theorem}
\begin{proof}
  Inclusion of a finite index subgroup induces a quasiisometry, so 
  \eqref{item:commraagtree}$\implies$\eqref{item:qiraagtree}.

  The graph of cylinders of a tree RAAG was described in the proof
  of \fullref{raags_are_basic}.
  Quasiisometry invariance of the JSJ tree of cylinders implies that
  any group quasiisometric to a tree RAAG has JSJ tree of cylinders
  whose cylinders are quasiisometric to $\free\times\mathbb{Z}$, whose
  rigid vertices are quasiisometric to $\mathbb{Z}^2$, has no hanging
  vertices, has exactly two cylinders adjacent to each rigid vertex,
  and such that in each cylinder the plane pattern coming from
  incident edges is a basic plane pattern.
  `Quasiisometric to' $\free\times \mathbb{Z}$ or $\mathbb{Z}^2$ can be
promoted to `virtually' $\free\times \mathbb{Z}$ or $\mathbb{Z}^2$,
respectively, since these groups are quasiisometrically rigid. 

Finally, if $\tilde{r}$ is a valence 2 rigid vertex of the JSJ tree of
cylinders and $r$ is its image in the graph of cylinders then $r$
either has two adjacent cylinders connected to $r$ via edge groups
equal to the vertex group of $r$, or it has one adjacent cylinder
connected to $r$ via an edge group that is an index 2 subgroup
of the vertex group of $r$.
The latter is impossible, by visibility of
the graph of cylinders in RACGs, since when $\Gamma$ is
triangle-free the only virtually $\mathbb{Z}^2$ special subgroups are
defined by squares of $\Gamma$, so one cannot be properly contained in another. 
Thus, \eqref{item:qiraagtree}$\implies$\eqref{item:jsjgoc}.

Assume \eqref{item:jsjgoc}.
Cylinders are suspensions
$\{a,b\}\join(\lk(a)\cap\lk(b))$ where there is an uncrossed cut $\{a-b\}$. 
By \eqref{item:cyclinderfreez}, the suspension
is non-square. 
Rigid vertices are maximal subsets of $\Gamma$ of at least four essential vertices that cannot be
separated by a cut.
By \eqref{item:rigidzsquared} 
they are virtually $\mathbb{Z}^2$, so they correspond
to squares in $\Gamma$ that are not separated by a cut.

A $K_{3,3}$ subgraph of $\Gamma$ cannot be separated by a cut.
A suspension
$\{a,b\}\join C$ with $|C|\geq 3$ can only be separated by a cut
of the form $\{a-b\}$, ie $\{a,b\}$ or $a\edge c\edge b$ for $c\in C$.
Thus the existence of either a maximal thick join
that is not a suspension or a non-square suspension whose suspension
points do not make a cut in one of these two ways implies there is a
non-virtually--$\mathbb{Z}^2$ rigid vertex, contradicting
\eqref{item:rigidzsquared}.
Thus, maximal thick joins are non-square suspensions whose pole gives a cut.
Since there are no hanging vertices, all cuts arise this way. 
Since the graph of cylinders gives a visual decomposition, every
vertex and edge of $\Gamma$ is contained in a subgraph corresponding
to a vertex of the graph of cylinders.
Since there are no hanging vertices and rigid vertices are assumed, by \eqref{item:rigidzsquared}, to
be contained in cylinders, every vertex and edge is contained in a
cylinder.
This establishes \eqref{item:cylinders}.

We have said rigid vertices correspond to certain squares and
cylinders correspond to maximal suspensions. A square is contained in a maximal suspension if and only if one of
its diagonals is the pole of the suspension.
Since a rigid vertex is assumed to be adjacent to two cylinders, both
diagonals of the corresponding square must be poles of maximal
suspensions.
This establishes \eqref{item:rigid}.

Item \eqref{item:nosuspensioncycle} describes a cylinder vertex of the graph of cylinders whose incident edges form a non-basic plane
pattern, contrary to \eqref{item:basicplanepattern}.

It remains to show \eqref{item:raagtreegraphconditions}
$\implies$\eqref{item:commraagtree}.
A graph $\Delta$ for which $A_\Delta$ is a finite index tree RAAG subgroup
of $W_\Gamma$ can be constructed essentially by taking the graph of
cylinders of $\Gamma$ and adding some leaves. 
We prove this by constructing a FIDL-$\Lambda$ with $\Delta$ as its
commuting graph, as in \fullref{thm:DL}.
This will take some work.

Assume \eqref{item:raagtreegraphconditions}.
Any graph of groups decomposition of $W_\Gamma$ has underlying graph a
tree because $W_\Gamma$ is generated by torsion elements, so it cannot
surject onto $\mathbb{Z}$. 
Let $T$ be the tree with one vertex for each cylinder in the graph of cylinders and an edge between cylinders if they share a square. 
This is just the underlying graph of the graph of cylinders, where we
have `forgotten' the rigid vertices, in the sense that they are all
of valence two, so we think of them as the midpoint of an edge instead
of as a valence two vertex.
In fact, we will think of $T$ as a subgraph of $\diag(\Gamma)$, whose
edges were defined to be the squares of $\Gamma$. 

We make some claims about the graph structure of $\Gamma$ deduced from
\eqref{item:raagtreegraphconditions}.
Assuming the claims, we finish the proof of \eqref{item:commraagtree}.
After that we will prove the claims. 

\begin{claim}\label{claim:everysquarehassuspension}
  For every square $\{a,b\}\join\{c,d\}$ of $\Gamma$, at least 
one of the diagonals $\{a,b\}$ or $\{c,d\}$ is the pole of a 
non-square suspension that is a maximal thick join.
\end{claim}
\begin{claim}\label{claim:tree}
  $\diag(\Gamma)$ is a tree and $T$ can be identified, via the
map sending a maximal thick join $\{a,b\}\join C$ to the vertex
$\{a,b\}\in\diag(\Gamma)$, with the subtree
  obtained by removing all leaves of $\diag(\Gamma)$,
\end{claim}

Examples illustrating \fullref{claim:suspendedsuspended} are shown in
\fullref{fig:claimsussus}.
In the claim, when $\{p_0,q_0\}$,$\{p_1,q_1\}$,\dots,$\{p_n,q_n\}$ is a
geodesic in $T\subset\diag(\Gamma)$ then for a vertex $v\in \Gamma$
let $I_v:=\{i\mid v\in\{p_i,q_i\}\}$.
Also, let $\even$ denote the even numbers, and $\odd$ the odd. 
\begin{claim}\label{claim:suspendedsuspended}
   If $\{p_0,q_0\}$,$\{p_1,q_1\}$,\dots,$\{p_n,q_n\}$ is a
geodesic in $T$ with $n>0$ and there exists a vertex $z\in\Gamma$  with $\{p_0,q_0,p_n,q_n\}\subset\lk(z)$ then:
\begin{enumerate}[(i)]
\item $n\in\even$\label{item:n_is_even}
\item $I_z=[0,n]\cap\odd$\label{item:z_is_odd}
  \item For all $0<i<n$, there is a cut $\{p_i-q_i\}$ separating
  $\{p_j,q_j\mid j<i\}\setminus \{p_i-q_i\}$ from  $\{p_j,q_j\mid
  i<j\}\setminus \{p_i-q_i\}$. If $i\in \even$ then the cut is
  $\{p_i-q_i\}=p_i\edge z\edge q_i$.\label{item:cuts}
\end{enumerate}
\end{claim}

\begin{figure}[h]
  \centering
  \begin{subfigure}{.45\textwidth}
      \centering
\begin{tikzpicture}[scale=2]\tiny
  \coordinate[label={[label distance=0pt] -90:$z$}] (z) at (0,-.5);
  \coordinate[label={[label distance=0pt] 180:$p_0$}] (p0) at (-.9,.1);
\coordinate[label={[label distance=0pt] 180:$q_0$}] (q0) at (-1.1,-.1);
\coordinate[label={[label distance=0pt] 90:$p_1$}] (p1) at (-0.5,0);
\coordinate[label={[label distance=0pt] -45:$$}] (q1) at (-0.5,0);
\coordinate[label={[label distance=0pt] 45:$p_2$}] (p2) at (.1,.1);
\coordinate[label={[label distance=0pt] -135:$q_2$}] (q2) at (-.1,-.1);
\coordinate[label={[label distance=0pt] 90:$p_3$}] (p3) at (.5,0);
\coordinate[label={[label distance=0pt] -45:$$}] (q3) at (.5,0);
\coordinate[label={[label distance=0pt] 0:$p_4$}] (p4) at (1.1,.1);
\coordinate[label={[label distance=0pt] 0:$q_4$}] (q4) at (.9,-.1);
\coordinate[label={[label distance=0pt] 180:$x$}] (x) at (-1,0);
\coordinate[label={[label distance=0pt] 0:$y$}] (y) at (1,0);
\draw (q0)--(q1)--(p0)--(p1)--cycle;
\draw (q1)--(q2)--(p1)--(p2)--cycle;
\draw (q2)--(q3)--(p2)--(p3)--cycle;
\draw (q3)--(q4)--(p3)--(p4)--cycle;
\draw (q0)--(z)--(p0) (q2)--(z)--(p2) (q4)--(z)--(p4);
\draw (p0)--(x)--(q0) (p4)--(y)--(q4); 
\filldraw (z) circle (.5pt);
\filldraw (p0) circle (.5pt);
\filldraw (q0) circle (.5pt);
\filldraw (p1) circle (.5pt);
\filldraw (p2) circle (.5pt);
\filldraw (q2) circle (.5pt);
\filldraw (p3) circle (.5pt);
\filldraw (p4) circle (.5pt);
\filldraw (q4) circle (.5pt);
\filldraw (x) circle (.5pt);
\filldraw (y) circle (.5pt);
\end{tikzpicture}
\caption{$\Gamma_1$, with $z=q_1=q_3$}
   \end{subfigure}
     \hfill
     \begin{subfigure}{.5\textwidth}
         \centering
\begin{tikzpicture}[scale=2]\tiny
    \coordinate[label={[label distance=-1pt] 180:$(p_0,p_2)$}] (p0p2) at (-1,.5);
  \coordinate[label={[label distance=-1pt] 180:$(p_0,q_2)$}] (p0q2) at (-1,.4);
  \coordinate[label={[label distance=-1pt] 180:$(q_0,p_2)$}] (q0p2) at (-1,.3);
  \coordinate[label={[label distance=-1pt] 180:$(q_0,q_2)$}] (q0q2) at (-1,.2);
  \coordinate[label={[label distance=-1pt,xshift=-2pt] -90:$(p_0,q_0)$}] (p0q0) at (-.8,0);
\coordinate[label={[label distance=-1pt,xshift=-1pt] -90:$(p_1,z)$}] (p1) at (-0.4,0);
\coordinate[label={[label distance=-1pt] -90:$(p_2,q_2)$}] (p2) at (0,0);
\coordinate[label={[label distance=-1pt] 0:$(p_1,p_3)$}] (p1p3) at (0,.5);
\coordinate[label={[label distance=-1pt,xshift=1pt] -90:$(p_3,z)$}] (p3) at (.4,0);
 \coordinate[label={[label distance=-1pt] 0:$(p_4,p_2)$}] (p4p2) at (1,.5);
  \coordinate[label={[label distance=-1pt] 0:$(p_4,q_2)$}] (p4q2) at (1,.4);
  \coordinate[label={[label distance=-1pt] 0:$(q_4,p_2)$}] (q4p2) at (1,.3);
  \coordinate[label={[label distance=-1pt] 0:$(q_4,q_2)$}] (q4q2) at (1,.2);
  \coordinate[label={[label distance=-1pt,xshift=2pt] -90:$(p_4,q_4)$}] (p4q4) at
  (.8,0);
    \coordinate[label={[label distance=-1pt] 180:$(x,p_1)$}] (xp1) at (-1,0);
  \coordinate[label={[label distance=-1pt] 180:$(x,z)$}] (xz) at (-1,.1);
  \coordinate[label={[label distance=-1pt] 0:$(y,z)$}] (yz) at (1,.1);
  \coordinate[label={[label distance=-1pt] 0:$(y,p_3)$}] (yp3) at (1,0);
  \draw[very thick, color=blue] (p1)--(p0q0)  (p2)--(p1) (p2)--(p3) (p3)--(p4q4) ;
  \draw (p1)--(p0p2) (p1)--(p0q2) (p1)--(q0p2)
(p1)--(q0q2);
\draw (p2)--(p1p3);
\draw (p3)--(p4p2) (p3)--(p4q2) (p3)--(q4p2)
(p3)--(q4q2);
\draw (xz)--(p0q0)--(xp1) (yz)--(p4q4)--(yp3);
\filldraw (p0p2) circle (.5pt);
\filldraw (p0q2) circle (.5pt);
\filldraw (q0p2) circle (.5pt);
\filldraw (q0q2) circle (.5pt);
\filldraw (p0q0) circle (.5pt);
\filldraw (p1) circle (.5pt);
\filldraw (p2) circle (.5pt);
\filldraw (p1p3) circle (.5pt);
\filldraw (p3) circle (.5pt);
\filldraw (p4p2) circle (.5pt);
\filldraw (p4q2) circle (.5pt);
\filldraw (q4p2) circle (.5pt);
\filldraw (q4q2) circle (.5pt);
\filldraw (p4q4) circle (.5pt);
\filldraw (xp1) circle (.5pt);
\filldraw (xz) circle (.5pt);
\filldraw (yz) circle (.5pt);
\filldraw (yp3) circle (.5pt);
\end{tikzpicture}
\caption{$\diag(\Gamma_1)$}
\end{subfigure}

\begin{subfigure}{.45\textwidth}
    \centering
\begin{tikzpicture}[scale=2]\tiny
  \coordinate[label={[label distance=0pt] 90:$z$}] (0) at (0,0.5);
\coordinate[label={[label distance=0pt] 180:$q_0$}] (1) at (-1.25,0.0);
\coordinate[label={[label distance=0pt] 180:$q_2$}] (2) at (-0.75,0.0);
\coordinate[label={[label distance=0pt] 180:$a$}] (3) at (-0.25,0.0);
\coordinate[label={[label distance=0pt] 0:$b$}] (4) at (0.25,0.0);
\coordinate[label={[label distance=0pt] 0:$p_6$}] (5) at (0.75,0.0);
\coordinate[label={[label distance=0pt] 0:$p_8$}] (6) at (1.25,0.0);
\coordinate[label={[label distance=0pt] -90:$q_1$}] (7) at (-0.75,-0.5);
\coordinate[label={[label distance=0pt] -90:$q_3$}] (8) at (-0.25,-0.5);
\coordinate[label={[label distance=0pt] -90:$q_5$}] (9) at (0.25,-0.5);
\coordinate[label={[label distance=0pt] -90:$q_7$}] (10) at
(0.75,-0.50);
\coordinate[label={[label distance=0pt] -90:$x$}] (x) at (-1.25,-0.50);
\coordinate[label={[label distance=0pt] -90:$y$}] (y) at (1.25,-0.50);
\draw (0)--(1);
\draw (0)--(2);
\draw (0)--(3);
\draw (0)--(4);
\draw (0)--(5);
\draw (0)--(6);
\draw (1)--(7);
\draw (2)--(7);
\draw (2)--(8);
\draw (3)--(7);
\draw (3)--(8);
\draw (3)--(9);
\draw (4)--(8);
\draw (4)--(9);
\draw (4)--(10);
\draw (5)--(9);
\draw (5)--(10);
\draw (6)--(10);
\draw (3)--(x)--(1) (4)--(y)--(6);
\filldraw (0) circle (.5pt);
\filldraw (1) circle (.5pt);
\filldraw (2) circle (.5pt);
\filldraw (3) circle (.5pt);
\filldraw (4) circle (.5pt);
\filldraw (5) circle (.5pt);
\filldraw (6) circle (.5pt);
\filldraw (7) circle (.5pt);
\filldraw (8) circle (.5pt);
\filldraw (9) circle (.5pt);
\filldraw (10) circle (.5pt);
\filldraw (x) circle (.5pt);
\filldraw (y) circle (.5pt);
\end{tikzpicture}
    \caption{$\Gamma_2$, with $z=p_1=p_3=p_5=p_7$, $a=p_0=p_2=p_4$, and $b=q_4=q_6=q_8$}
     \end{subfigure}
     \hfill
     \begin{subfigure}{.5\textwidth}
         \centering
\begin{tikzpicture}[scale=2]\tiny
  \coordinate[label={[label distance=-1pt] -90:$(z,q_1)$}] (0) at (-0.75,-0.0);
\coordinate[label={[label distance=-1pt] 90:$(q_0,q_2)$}] (1) at (-0.75,0.25);
\coordinate[label={[label distance=-1pt] 90:$(a,q_0)$}] (2) at (-1,-0);
\coordinate[label={[label distance=-1pt] -90:$(z,q_3)$}] (3) at (-0.25,-0.0);
\coordinate[label={[label distance=-1pt] 90:$(a,q_2)$}] (4) at (-0.5,-0.0);
\coordinate[label={[label distance=-1pt] 90:$(b,q_2)$}] (5) at (-0.25,0.250);
\coordinate[label={[label distance=-1pt] 90:$(a,b)$}] (6) at (0.0,-0.0);
\coordinate[label={[label distance=-1pt] -90:$(z,q_5)$}] (7) at (0.25,-0.0);
\coordinate[label={[label distance=-1pt] 90:$(a,p_6)$}] (8) at (0.25,0.25);
\coordinate[label={[label distance=-1pt] 90:$(p_6,b)$}] (9) at (0.5,-0.0);
\coordinate[label={[label distance=-1pt] -90:$(z,q_7)$}] (10) at (0.75,-0.0);
\coordinate[label={[label distance=-1pt] 90:$(p_8,b)$}] (11) at (1,-0.0);
\coordinate[label={[label distance=-1pt] 90:$(p_6,p_8)$}] (12) at (.75,0.250);
\coordinate[label={[label distance=-1pt] -90:$(q_1,q_3)$}] (13) at (-0.5,-0.250);
\coordinate[label={[label distance=-1pt] -90:$(q_3,q_5)$}] (14) at (0,-0.25);
\coordinate[label={[label distance=-1pt] -90:$(q_5,q_7)$}] (15) at
(0.5,-0.25);
\coordinate[label={[label distance=-1pt] -90:$(x,z)$}] (16) at (-1,-0.25);
\coordinate[label={[label distance=-1pt] 0:$(x,q_1)$}] (17) at (-.95,-0.2);
\coordinate[label={[label distance=-1pt] -90:$(y,z)$}] (18) at (1,-0.25);
\coordinate[label={[label distance=-1pt] 180:$(y,q_7)$}] (19) at (0.95,-0.2);
\draw (0)--(1);
\draw[very thick,color=blue] (2)--(0)--(4)--(3)--(6)--(7)--(9)--(10)--(11);
\draw (3)--(5);
\draw (4)--(13);
\draw (6)--(14);
\draw (7)--(8);
\draw (9)--(15);
\draw (10)--(12);
\draw (16)--(2)--(17) (18)--(11)--(19);
\filldraw (0) circle (.5pt);
\filldraw (1) circle (.5pt);
\filldraw (2) circle (.5pt);
\filldraw (3) circle (.5pt);
\filldraw (4) circle (.5pt);
\filldraw (5) circle (.5pt);
\filldraw (6) circle (.5pt);
\filldraw (7) circle (.5pt);
\filldraw (8) circle (.5pt);
\filldraw (9) circle (.5pt);
\filldraw (10) circle (.5pt);
\filldraw (11) circle (.5pt);
\filldraw (12) circle (.5pt);
\filldraw (13) circle (.5pt);
\filldraw (14) circle (.5pt);
\filldraw (15) circle (.5pt);
\filldraw (16) circle (.5pt);
\filldraw (17) circle (.5pt);
\filldraw (18) circle (.5pt);
\filldraw (19) circle (.5pt);
\end{tikzpicture}
         \caption{$\diag(\Gamma_2)$}
     \end{subfigure}
  \caption{Two graphs as in \fullref{claim:suspendedsuspended} with
    their diagonal graphs containing the geodesic
    $\{p_0,q_0\}, \{p_1,q_1\},\dots,\{p_n,q_n\}$.}
  \label{fig:claimsussus}
\end{figure}

\begin{claim}\label{claim:evendistance}
If $\{a,b\}\neq\{a,c\}$ are vertices in
$\diag(\Gamma)$ then the distance between them is even and
every second vertex on the geodesic between them contains $a$. 
\end{claim}
\begin{claim}\label{claim:uniquesuspension}
  Every vertex of $\Gamma$ is suspended in some maximal suspension.
  A vertex that is not a suspension point of any maximal suspension is contained in a unique maximal suspension.
\end{claim}
\medskip

Assuming the claims, we construct a FIDL--$\Lambda$.

By \fullref{claim:tree}, $\diag(\Gamma)$ is a tree, so it can be
2--colored 0/1 such that neighboring vertices have different colors.
By \fullref{claim:evendistance}, if $v$ is in the support of two
different vertices of $\diag(\Gamma)$ then they have even distance in
$\diag(\Gamma)$, so they have the same color.
Thus, the 2--coloring on $\diag(\Gamma)$ induces a 2--coloring of
$\Gamma$. 

Let $\{a,b\}$ be a vertex of $T$, corresponding to a maximal
suspension $\{a,b\}\join\{c_0,\dots,c_{n-1}\}$.
Let $\Lambda_{a,b}\subset \Gamma^c$ be a graph on vertices
$\{c_0,\dots,c_{n-1}\}$, constructed as follows.
Add an edge in $\Lambda_{a,b}$ between $c_i$ and $c_j$ if
$\{c_i,c_j\}$ is a vertex of $T$.
Call these `mandatory edges'. 
Condition \eqref{item:nosuspensioncycle} implies the mandatory edges
form a forest.
Add additional `discretionary' edges to make  $\Lambda_{a,b}$ a tree. 
Let $T_i\subset T$ be the vertices with color $i$.
Let $\Lambda_0$ be the union of trees $\Lambda_{a,b}$ such that
$\{a,b\}\in T_1$, so that $\supp(\Lambda_0)$ consists of vertices
of $\Gamma$ colored 0. 
Define $\Lambda_1$ analogously for $T_0$.
By \fullref{claim:uniquesuspension},
$\Lambda:=\Lambda_0\sqcup\Lambda_1$ spans $\Gamma$.
Vertices of $\Lambda_i$ are colored $i$, so they are not adjacent in
$\Gamma$. 

We prove $\Lambda_0$ is a tree by induction. 
Pick a vertex $\{a_0,b_0\}\in T_1$.
By construction,  $\Lambda_{a_0,b_0}$ is a tree.
Now suppose that $\bigcup_{i=0}^{m-1}\Lambda_{a_i,b_i}$ is a tree,
where the set of vertices $\{a_i,b_i\}\in T_1$ with $i<m$ is the intersection of a convex subset of
$T$ with $T_1$.
Let $\{a_{m},b_{m}\}\in T_1$ that is
closest to the existing set.
Since $\diag(\Gamma)$ is a tree, the convexity condition implies that there is a unique vertex $\{c,d\}\in T_0$ separating $\{a_{m},b_{m}\}$ from all of the
other $\{a_i,b_i\}$, and there is at least one index $j<m$ such that
$\{c,d\}$ is a neighbor of $\{a_{j},b_{j}\}$ in $T$.
Thus, $c\edge d$ is a mandatory edge in $\Lambda_{a_{m},b_{m}}$ and
in any such $\Lambda_{a_{j},b_{j}}$.
We have $c\edge d\subset \Lambda_{a_m,b_m}\cap\bigcup_{i=0}^{m-1}\Lambda_{a_i,b_i}$.
Conversely, suppose, for some $j<m$, that
$\Lambda_{a_{j},b_{j}}\cap\Lambda_{a_{m},b_{m}}\neq
\emptyset$, so there is a vertex $z\in\Gamma$ with $\{a_{j},b_{j},a_{m},b_{m}\}\subset\lk(z)$.
Since $\{c,d\}$ is the penultimate vertex on the unique geodesic from
$\{a_{j},b_{j}\}$ to $\{a_{m},b_{m}\}$ in $\diag(\Gamma)$,
\fullref{claim:suspendedsuspended} says $z\in\{c,d\}$.
Thus, $\Lambda_{a_m,b_m}\cap\bigcup_{i=0}^{m-1}\Lambda_{a_i,b_i}$ is
exactly the edge $c\edge d$, but the  union of two trees with exactly
an edge in common is a tree, so $\Lambda_0$ is a tree by induction,
and $\Lambda_1$ similarly. 

Now we check that conditions $\mathcal{R}_3$  and $\mathcal{R}_4$ of
\fullref{thm:DL} are satisfied.

Suppose $\{a,b\}\join\{c,d\}$ is a
square in $\Gamma$.
\fullref{claim:everysquarehassuspension} says either $\{a,b\}$ or
$\{c,d\}$ is the pole of
a maximal suspension; suppose $\{a,b\}\join C$ is a maximal suspension
with $\{c,d\}\subset C$.
Then $\hull_\Lambda\{a,b\}=\{a,b\}$ and
$\hull_\Lambda\{c,d\}\subset\Lambda_{a,b}\subset C$, so 
 $\hull_\Lambda\{a,b\}\join\hull_\Lambda\{c,d\}\subset
\{a,b\}\join C\subset \Gamma$.
Condition $\mathcal{R}_3$ is satisfied. 

Suppose condition $\mathcal{R}_4$ is not satisfied.
Then there is a shortest simple cycle $\gamma$ in $\Gamma$ containing an edge $e$ that
is not in a square with vertices in $\hull_\Lambda(\gamma)$.
Such a $\gamma$ is induced, because if $\gamma$ contains vertices $v$
and $w$ that
are adjacent in $\Gamma$ but not in $\gamma$ then we could surger
$\gamma$ into two simple cycles $\gamma'$ and $\gamma''$ that are each
an arc of $\gamma$ plus the edge $v\edge w$.
Both are strictly shorter than $\gamma$, and one of them, say $\gamma'$,
contains $e$, but
$\hull_\Lambda(\gamma')\subset\hull_\Lambda(\gamma)$, so $e$ is not
contained in a square with vertices in $\hull_\Lambda(\gamma')$.
This would contradict minimality of $\gamma$.

Let $u\edge v\edge w\edge x$ be the subsegment of $\gamma$ such that
$e=v\edge w$.
If $u\edge v$ and $w\edge x$ are contained in a common suspension then
so is $u\edge x$, and $\{u,w\}\join\{v,x\}$ is a square, contradicting the definition of
$e$.
Thus, we may assume that there is not a suspension containing
all three of the edges.
The argument will be to produce a cut through $v$ or $w$ separating $u$
from $x$.
Since $\gamma$ is a loop, it must pass back through the cut again at
some other vertex $v'$, which will lead to a contradiction with the
assumption that $\gamma$ is induced.
Thus, there is no witness to the failure of $\mathcal{R}_4$.

{\bf Case 1:} 
There is a maximal suspension
$\sigma$ containing  $u\edge v\edge w$ and a different maximal
suspension $\sigma'$ containing $v\edge  w\edge x$.
Suppose the distance between $\sigma$ and $\sigma'$ in $T$ is minimal
among pairs with these properties. 

{\bf Case 1a:}
Suppose $\sigma$ and $\sigma'$ are adjacent in
$T$.
Their intersection is the join of their poles.
If $\{u,w\}$ is the pole of $\sigma$ then all four of $u$, $v$, $w$,
and $x$ are contained in $\sigma'$, contrary to hypothesis, and
similarly for $\sigma'$, so  there are vertices $v'\neq w'$ such that $\sigma\cap\sigma'=\{v,v'\}\join\{w,w'\}$.
If no cut separates $u$ from $x$ then
$\{u\}\join\{v,v'\}\join\{w,w'\}\join\{x\}$ is a non-square rigid
vertex, which is impossible, so there is at least one cut of the form $\{v-v'\}$ or
$\{w-w'\}$ separating $u$ from $x$.
These cases are symmetric, so assume it is $\{v-v'\}$.
Since $v$ is adjacent to $u$ and $w$, if $\gamma$ passes through $v'$
then $\gamma$ fails to be induced.
However, the only way for $\gamma$ to pass back through the cut
avoiding $v'$ is if the cut is a cut 2--path $v\edge
u'\edge v'$ with $u'\neq w$.
In this case, $\gamma$ passes back through $u'$, which is adjacent to
$v$, so $\gamma$ is not induced. 

{\bf Case 1b:}  Suppose $\sigma$ and $\sigma'$ are not adjacent in $T$.
For $\{a,b\}\in T$, let
$\sigma_{a,b}:=\{a,b\}\join(\mathrm{link}(a)\cap\mathrm{link}(b))$ be
the maximal suspension of $\Gamma$ with pole $\{a,b\}$. 
Let $\{p_0,q_0\}$, \dots, $\{p_n,q_n\}$ be the geodesic in $T$ between
$\sigma=\sigma_{p_0,q_0}$ and $\sigma'=\sigma_{p_n,q_n}$.
Vertices $v$ and $w$ cannot both be suspended in the same suspension since they
are connected by an edge and $\Gamma$ is triangle free, so, 
  up to symmetry, there are two subcases: either $w$ is a suspended
  and $v$ is suspension in both $\sigma$ and $\sigma'$, or $v$ is a
  suspension point of $\sigma$ and suspended in $\sigma'$ and the
  opposite is true for $w$.

  In the first subcase, $\{p_0,q_0\}=\{v,v'\}$, for some $v'$, and $\{p_n,q_n\}=\{v,x\}$. 
 \fullref{claim:suspendedsuspended} says $n\in\even$,
 $w\in\{p_1,q_1\}$, and there is a cut $\{p_1-q_1\}$ separating
 $\{v,v'\}\setminus \{p_1-q_1\}$ from $\{v,x\}\setminus\{p_1-q_1\}$.
 Since $v\in\{v,v'\}\cap\{v,x\}$, the cut is the cut 2--path
 $p_1\edge v\edge q_1$, and it separates $v'$ from $x$.
 It also separates $u$ from $x$, since $u$ is adjacent to $v'$.
 So, $\gamma$ crosses the cut once through $v$  to get from $u$ to $x$, and must
 cross back again later, but both of the other vertices of the cut are
 adjacent to $v$,
 contradicting that $\gamma$ is induced.

 In the second subcase we have that $w$ is suspended in $\sigma_{p_{n-1}q_{n-1}}$, and, 
 by minimality of $n$, $u\notin \sigma_{p_1,q_1}$ and $x\notin\sigma_{p_{n-1},q_{n-1}}$.
 Apply \fullref{claim:suspendedsuspended} to
 $\{p_0,q_0\}$,\dots,$\{p_{n-1},q_{n-1}\}$ with $z=w$, which gives that
 $w\in\{p_1,q_1\}$ and there is a cut $\{p_1-q_1\}$ in $\Gamma$ separating
 $\{p_0,q_0\}\setminus \{p_1-q_1\}$ from $\{p_{n-1},q_{n-1}\}\setminus
 \{p_1-q_1\}$.
 Since $u\notin \sigma_{p_1,q_1}$ and
 $x\notin\sigma_{p_{n-1},q_{n-1}}$, $\{p_1-q_1\}$ separates $u$ from
 $x$.
 Since $v$ is adjacent to $u$ and $v$ and $x$ are both suspended in
 $\sigma'$ with two common neighbors $\{p_n,q_n\}\neq \{p_1,q_1\}$,
 the cut $\{p_1-q_1\}$ is a cut 2--path $p_1\edge
 v\edge q_1$.
 Thus, $\gamma$ crosses $\{p_1-q_1\}$ once through $v$ and then back
 again through a vertex adjacent to $v$, contradicting that it is induced.
 
{\bf Case 2:} The path $u\edge v\edge w$ is not
contained in a maximal suspension.   
Let $\{p_0,q_0\}$,\dots,$\{p_n,q_n\}$ be a geodesic in $T$ such that $u,v\in \sigma_{p_0,q_0}$ and
  $v,w\in\sigma_{p_n,q_n}$, and such that $n>0$ is minimal among such
  geodesics.
  Edges sharing a common vertex cannot be contained in adjacent
  maximal suspensions without both being contained in at least one of the
  two, so $n\geq 2$. 
  Minimality implies $|\{u,v\}\cap \{p_0,q_0\}|=1$ and $\{u,v\}\cap\{p_1,q_1\}=\emptyset$ and
  $|\{v,w\}\cap \{p_n,q_n\}|=1$ and
  $\{v,w\}\cap\{p_{n-1},q_{n-1}\}=\emptyset$.
  
 If $v$ is suspended in both $\sigma_{p_0,q_0}$ and
   $\sigma_{p_n,q_n}$ then \fullref{claim:suspendedsuspended} says
   $v\in\{p_1,q_1\}$, but this contradicts minimality of $n$.

  If $v$ is suspended in $\sigma_{p_0,q_0}$ but not in 
    $\sigma_{p_n,q_n}$
  then $v\in\{p_n,q_n\}$ is adjacent to both $p_{n-1}$ and $q_{n-1}$.
  Apply \fullref{claim:suspendedsuspended} to the geodesic subsegment
  from 
  $\{p_0,q_0\}$ to $\{p_{n-1},q_{n-1}\}$. Again, this implies
  $v\in\{p_1,q_1\}$, contradicting minimality of $n$.
  The symmetric argument works if $v$ is suspended in
  $\sigma_{p_n,q_n}$ but not in $\sigma_{p_0,q_0}$.

   Finally, suppose $v$ is a suspension point in both  $\sigma_{p_0,q_0}$ and
   $\sigma_{p_n,q_n}$.
    If $n=2$ then apply \fullref{claim:suspendedsuspended} with $z=p_1$.
   If $n>2$ then 
   apply \fullref{claim:suspendedsuspended} to the geodesic subsegment
   $\{p_1,q_1\}$,\dots,$\{p_{n-1},q_{n-1}\}$ with $z=v$.
   In both cases we get that $p_{n-1}\edge v\edge q_{n-1}$ is a
   cut 2--path, and 
the loop $\gamma$ contains $v$ and  has vertices $u$ and $w$ on different sides of the
cut, so it must cross the cut a second time in either $p_{n-1}$ or
$q_{n-1}$, both of which are adjacent to $v$, contradicting that $\gamma$ is induced. 

\medskip

Up to symmetry, this accounts for all possibilities, so we have
constructed a FIDL--$\Lambda$.
By \fullref{thm:DL}, the commuting graph $\Delta$ of $\Lambda$ defines
an index 4 visible RAAG subgroup of $W_\Gamma$.
It follows from Behrstock and Neumann's \cite{BehNeu08} work that
$\Delta$ must be a tree, because tree RAAGs are not quasiisometric to
nontree RAAGs, but we can say explicitly what $\Delta$ is from the
construction: it is a tree built by starting from $T$ and at each
vertex $\{a,b\}\in T$ adding a new leaf for each discretionary
edge of $\Lambda_{a,b}$. 
This completes the proof of \eqref{item:raagtreegraphconditions}
$\implies$\eqref{item:commraagtree}, modulo the claims.

\begin{claimproof}[\ref{claim:everysquarehassuspension}]
Any square  $\{a,b\}\join\{c,d\}$ in $\Gamma$ is a thick join, so it
is contained in some maximal thick join $\sigma$, which,
by \eqref{item:cylinders} is a non-square suspension.
Since $\Gamma$ is triangle-free either $\{a,b\}$ or $\{c,d\}$ is the
pole of $\sigma$.
\end{claimproof}

\begin{claimproof}[\ref{claim:tree}]
  If $\{a,b\}\join\{c,d\}$ in $\Gamma$ is a square and 
 $\{a,b\}\join\{x,y\}$ is another square then $|\{c,d,x,y\}|\geq
3$.
The thick join $\{a,b\}\join\{c,d,x,y\}$ is contained in a maximal thick join, which is a suspension, by \eqref{item:cylinders}, so $\{a,b\}$ is the pole of a
maximal thick join that is a non-square suspension.

Every non-leaf of $\diag(\Gamma)$ is a pair of vertices of $\Gamma$ that are the
diagonal of more than one square, so, as above, is the pole of a
non-square suspension.
These correspond to vertices of $T$, so all of the non-leaf vertices
of $\diag(\Gamma)$ are identified with vertices of $T$.
Furthermore, adjacent vertices of $\diag(\Gamma)$ span a square, so
$T$--vertices of $\diag(\Gamma)$ are adjacent in $\diag(\Gamma)$ if
and only if they are adjacent in $T$. 
Thus, we identify $T$ with a subgraph of $\diag(\Gamma)$ that includes
all non-leaves of $\diag(\Gamma)$.
By \fullref{claim:everysquarehassuspension}, $\diag(\Gamma)$ has no
isolated edges, so every leaf is attached to one of the $T$--vertices of
$\diag(\Gamma)$.
Thus, $\diag(\Gamma)$ is a tree and $T$ is a subtree containing all
non-leaves.  
 
For $\{a,b\}\in T$ there is a maximal thick join $\{a,b\}\join
\{c_0,c_1,c_2,\dots\}$ in $\Gamma$.
Since $\Gamma$ is triangle free, there are induced squares
$\{a,b\}\join\{c_i,c_j\}$ for all $i\neq j$. 
Thus, $\{a,b\}$ has at least
three neighbors in $\diag(\Gamma)$.
In particular, it is not a leaf.
\end{claimproof}

\begin{claimproof}[\ref{claim:suspendedsuspended}]
  The proof is by induction on $n$.
  Since $\Gamma$ is triangle-free, $n>1$.
  Suppose $n=2$.
  Since $\{p_0,q_0\}$ and
  $\{p_2,q_2\}$ are distinct vertices in $T$,
  $|\{p_0,q_0,p_2,q_2\}|\geq 3$, so
$\{p_0,q_0,p_2,q_2,z\}$ is a set of at
  least 4 essential vertices that is not a square, since $z$ is
  adjacent to at least three of the others.
  By \eqref{item:rigid}, the set must be separated by a cut, since otherwise there would be a
  non-(virtually $\mathbb{Z}^2$) rigid vertex in the graph of
  cylinders.
  By \eqref{item:cylinders}, for all $i$, vertices $p_i$ and $q_i$ have at
  least three common neighbors, so they cannot be separated by a cut.
  Now,  $\{z,p_1,q_1\}\join\{p_0,q_0,p_2,q_2\}\subset\Gamma$, but thick joins in $\Gamma$ are suspensions,  by \eqref{item:cylinders}, so $|\{z,p_1,q_1\}|=2$. 
  Thus, $z\in\{p_1,q_1\}$, and there is a cut $\{p_1-q_1\}$ separating
  $\{p_0,q_0\}\setminus\{p_1-q_1\}$ from $\{p_2,q_2\}\setminus\{p_1-q_1\}$.
  This proves the claim when $n=2$.

  Now consider a geodesic of length $n>2$ as in the statement of the
  claim, and suppose the claim is true for all geodesics in $T$ of
  length strictly less than $n$.
  As in the base case, for all $i$ there is no cut separating $p_i$
  from $q_i$, but there must be a cut separating
  $\{p_0,q_0,p_n,q_n,z\}$, and the cut must contain $z$. 
  Thus, there is some cut
  containing $z$ and separating a subset of $\{p_0,q_0\}$ from a
  subset of $\{p_n,q_n\}$. 
  The geodesic $\{p_0,q_0\}$,\dots,$\{p_n,q_n\}$ in $T$ gives a chain of
  squares in $\Gamma$ with adjacent squares sharing a diagonal.
  Such a set is 2--connected, and the only potential cut
  pairs are $\{p_i,q_i\}$ for $0<i<n$.
  Thus:
  \begin{equation}
    \label{eq:d}
    \begin{minipage}{.9\textwidth}
      There exists $0<i_0<n$ such that there is a cut $\{p_{i_0}-q_{i_0}\}\supset\{p_{i_0},q_{i_0},z\}$ separating the set $\{p_0,q_0,p_n,q_n,z\}$.
  \end{minipage}
 \tag{$\dagger$}
  \end{equation}
  
The goal will be to split the $T$--geodesic into two subsegments at
index $i_0$
and apply the induction
hypothesis to both sides. 
We will also use the following fact:
\begin{equation}
  \label{eq:dd}
  \begin{minipage}{.9\textwidth}
     No cut of the form $\{p_i-q_i\}$  separates the set
  $\{p_j,q_j\mid j<i\}\setminus \{p_i-q_i\}$ or separates the set
  $\{p_j,q_j\mid j>i\}\setminus \{p_i-q_i\}$.
    \end{minipage}
 \tag{$\ddagger$}
\end{equation}
To see \eqref{eq:dd}, first note that for all $j<i$,
$\{p_j,q_j\}\nsubset\{z,p_i,q_i\}$.
This is true because $\{p_j,q_j\}\neq\{p_i,q_i\}$, and $p_j$ and $q_j$
are not adjacent, but $z$ is adjacent to both of $p_i$ and $q_i$.
Thus, for all $j<i$ at least one of $p_j$ or $q_j$ survives in
$\{p_j,q_j\mid j<i\}\setminus \{p_i-q_i\}$ and any vertices with
consecutive indices are adjacent. 

The combination of \eqref{eq:d} and \eqref{eq:dd} implies there exists $0<i_0<n$ such that there is a cut $\{p_{i_0}-q_{i_0}\}\supset\{p_{i_0},q_{i_0},z\}$ 
separating $\{p_i,q_i\mid i<i_0\}\setminus
\{p_{i_0}-q_{i_0}\}$ from $\{p_i,q_i\mid i>i_0\}\setminus
\{p_{i_0}-q_{i_0}\}$.
Consider two cases based on whether $z$ is in $\{p_{i_0},q_{i_0}\}$.

{\bf Case $z\notin\{p_{i_0},q_{i_0}\}$:}
Then the cut $\{p_{i_0}-q_{i_0}\}$ is a cut 2--path $p_{i_0}\edge
z\edge q_{i_0}$.
Consider the geodesic subsegments
$\{p_0,q_0\}$,\dots,$\{p_{i_0},q_{i_0}\}$ and
$\{p_{i_0},q_{i_0}\}$,\dots,$\{p_{n},q_{n}\}$.
Each is strictly shorter than the one we started with and has all
vertices in the supports of its two endpoints contained in $\lk(z)$.
Apply the induction hypothesis.
Then $i_0\in \even$ and $n-i_0\in\even$, so $n\in \even$.
We have the even number $i_0$ is not in $I_z$, and by induction $z$
occurs in the support of each odd index of each geodesic subsegment,
so $I_z=[0,n]\cap\odd$.
Finally, the induction hypothesis for the first subsegment says that
for every $i<i_0$ there is a cut $\{p_i-q_i\}$ separating
$\{p_j,q_j\mid 0\leq j<i\}\setminus\{p_i-q_i\}$ from $\{p_j,q_j\mid
i<j\leq i_0\}\setminus\{p_i-q_i\}$, but \eqref{eq:dd} says
$\{p_i-q_i\}$ does not separate $\{p_j,q_j\mid
i<j\leq n\}\setminus\{p_i-q_i\}$, so it separates $\{p_j,q_j\mid 0\leq j<i\}\setminus\{p_i-q_i\}$ from $\{p_j,q_j\mid
i<j\leq n\}\setminus\{p_i-q_i\}$.
A symmetric argument on the second geodesic subsegment finishes the
proof of \ref{item:cuts}.

{\bf Case $z\in\{p_{i_0},q_{i_0}\}$:} Without loss of generality,
assume $z=p_{i_0}$.
First suppose $1<i_0<n-1$, and consider the geodesic subsegments $\{p_0,q_0\}$,\dots,$\{p_{i_0-1},q_{i_0-1}\}$ and
$\{p_{i_0+1},q_{i_0+1}\}$,\dots,$\{p_{n},q_{n}\}$.
Again, each is strictly shorter than the one we started with and has
all vertices in the support of its two endpoints contained in
$\lk(z)$, so induct.
We get $i_0-1\in\even$ and $n-(i_0-1)\in\even$, so $n\in\even$ and
$i_0\in\odd$.
We have that $I_z$ contains the odd index $i_0$, since $z=p_{i_0}$, and the
induction hypothesis for the two subsegments gives that $I_z$ contains
every other odd index as well.
The proof of \ref{item:cuts} is the same as in the previous case.

If $i_0=1$ or $i_0=n-1$ the proof is the same, except one of the
geodesic subsegments is a single vertex, so the induction is only
necessary on the other. 
\end{claimproof}

\begin{claimproof}[\ref{claim:evendistance}]
  By \fullref{claim:tree}, $\diag(\Gamma)$ is a tree, so there is a
  unique geodesic  $\{a,b\}=\{p_0,q_0\}$,\dots,$\{p_n,q_n\}=\{a,c\}$.
  The support of a single edge has 4 vertices, so $n>1$.
  If $n=2$ the claim is true.
  We cannot have $n=3$, because then $\{p_1,q_1\}\join\{p_2,q_2\}$ is
  a square in $\Gamma$ contained in $\lk(a)$, contradicting that
  $\Gamma$ is triangle-free. 
  If $n>3$, apply \fullref{claim:suspendedsuspended} to
  $\{p_1,q_1\},\dots,\{p_{n-1},q_{n-1}\}$ with $z=a$.
\end{claimproof}

\begin{claimproof}[\ref{claim:uniquesuspension}]
  Since $\Gamma$ is triangle-free, the intersection of suspensions
  corresponding to adjacent vertices in $T$ is exactly the square
  formed by the join of the two poles.
  By \eqref{item:cylinders}, every vertex of $\Gamma$ is contained
  in some maximal suspension.
  If $v\in\Gamma$ is a suspension point of a maximal suspension
  $\sigma$, then it is suspended in each neighbor of $\sigma$ in $T$,
  so every vertex of $\Gamma$ is suspended in some suspension. 

  If $z\in \sigma\cap\sigma'$ is contained in two maximal
  suspensions, but is not in the pole of either, then $\sigma$ and
  $\sigma'$ do not give adjacent vertices in $T$. 
\fullref{claim:suspendedsuspended} says $z$ occurs in the support of
each odd index vertex on the $T$-geodesic from $\sigma$ to $\sigma'$.
Thus, every vertex that is contained in distinct maximal suspensions
is in the pole of some maximal suspension. 
\end{claimproof}
\end{proof}

\begin{example}\label{ex:nonplanar_raag_tree_redux}
 The graph of \fullref{fig:nonplanar_raag_tree}
  is shown again with its graph of cylinders in
  \fullref{fig:nonplanar_raag_tree_redux}.
  The hypotheses of \fullref{maintheorem} are satisfied,
  because the central cylinder vertex has basic plane pattern:
  $W_{\{3,4\}\join\{0,1,2,8\}}$ has an index 4 subgroup $\langle
  w,x,y\rangle\times\langle z\rangle\cong \free_3\times \mathbb{Z}$ given
  by $w:=s_0s_8$, $x:=s_1s_8$, $y:=s_2s_8$, and $z:=s_3s_4$, such that
  restriction to the subgroup induces a pattern preserving
  quasiisometry to $\free_3\times\mathbb{Z}$ with plane pattern generated
  by   $\{\langle w\rangle\times\langle z\rangle,\, \langle
  x\rangle\times\langle z\rangle,\, \langle y\rangle\times\langle
  z\rangle\}$, so the projected line pattern is generated by $\{w,x,y\}$ in $\langle
  w,x,y\rangle$, which is basic.
  \begin{figure}[h]
  \begin{subfigure}{.35\textwidth}
      \centering
\begin{tikzpicture}[scale=.5]\tiny
    \coordinate[label={[label distance=0pt] -90:$4$}]  (0) at (0,0);
    \coordinate[label={[label distance=0pt] 45:$8$}]  (1) at (0,1);
    \coordinate[label={[label distance=0pt] 90:$3$}]  (2) at (0,2);
    \coordinate[label={[label distance=0pt] 180:$2$}]  (3) at (-2,2);
    \coordinate[label={[label distance=0pt] 180:$1$}]  (4) at (-2,0);
    \coordinate[label={[label distance=0pt] 0:$0$}]  (5) at (2,1);
    \coordinate[label={[label distance=0pt] 90:$7$}]  (6) at (-1,1.5);
    \coordinate[label={[label distance=-2pt] -45:$6$}]  (7) at (-1,.5);
    \coordinate[label={[label distance=0pt] 90:$5$}]  (8) at (1,1);

    \filldraw (0) circle (2pt) (1) circle (2pt) (2) circle (2pt) (3)
    circle (2pt) (4) circle (2pt) (5) circle (2pt) (6) circle (2pt)
    (7) circle (2pt) (8) circle (2pt);
    \draw (0)--(1)--(2)--(4)--(7)--(1)--(8)--(5)--(0)--(3)--(2)
    (3)--(6)--(1) (0)--(4) (2)--(5);
  \end{tikzpicture}
      \subcaption{$\Gamma$}
  \label{fig:nonplanar_raag_tree_numbered}
\end{subfigure}
  \begin{subfigure}{.6\textwidth}
\centering
  \begin{tikzpicture}\tiny
 \coordinate[label={[label distance=0pt] 180:$W_{\{3,4\}}\times
   W_{\{0,1,2,8\}}$}]  (0) at (0,0);
 \coordinate[label={[label distance=0pt] 0:$W_{\{0,8\}}\times W_{\{3,4,5\}}$}]  (1) at (0:1);
\coordinate[label={[label distance=0pt] 0:$W_{\{1,8\}}\times W_{\{3,4,6\}}$}]  (2) at (-120:1);
  \coordinate[label={[label distance=0pt] 0:$W_{\{2,8\}}\times
    W_{\{3,4,7\}}$}]  (3) at (120:1);
  \draw (0)--(1) (0)--(2) (0)--(3);
  \filldraw (0) circle (1pt) (1) circle (1pt) (2) circle (1pt) (3)
  circle (1pt) ($(0)!.5!(1)$) circle (1pt) ($(0)!.5!(2)$) circle (1pt) ($(0)!.5!(3)$) circle (1pt) ;
\end{tikzpicture}
  \subcaption{ graph of cylinders of $W_\Gamma$.}
  \label{fig:jsjgoc_nonplanar_raag_tree_numbered}
\end{subfigure}
\caption{Graph and graph of cylinders for
  \fullref{ex:nonplanar_raag_tree_redux}.}
\label{fig:nonplanar_raag_tree_redux}
\end{figure}
\end{example}
\section{Further remarks}
The implication \eqref{item:jsjgoc} $\implies$ \eqref{item:qiraagtree}
of \fullref{maintheorem}
can be proven directly as in \cite{Cas10}.
The description of the JSJ decomposition
means these groups, like the tree RAAGs, are `two-line tubular groups
with bounded height change', which are all quasiisometric
\cite[Example~5.1]{Cas10}.
The implication \eqref{item:raagtreegraphconditions} $\implies$ \eqref{item:jsjgoc} 
is also easy.
It should be possible to generalize Edletzberger's
\cite{Edl24} visual description of the graph of cylinders in the
case that $\Gamma$ has triangles.
In this case one could hope to replace the `no-triangles' hypothesis on
$\Gamma$ by `no-icosahedra' and retain the equivalence between
\eqref{item:qiraagtree}, \eqref{item:jsjgoc}, and  \eqref{item:raagtreegraphconditions}.
Graphs with no icosahedra are the most general family for which such a
classification would be interesting, because if $\Gamma$ contains an
icosahedron then $\mathbb{Z}^3\leq W_\Gamma$, so $W_\Gamma$ is not
quasiisometric to a tree RAAG. 

The equivalence of \eqref{item:commraagtree} to the others
if the `no-triangles'
hypothesis is relaxed is another matter. 
 \fullref{thm:DL} is an application of conditions of Dani and
 Levcovitz \cite{DanLev24} specialized to the triangle-free case.
Their work partially handles more general situations, but when there are triangles there are more conditions and they are only
 necessary, not sufficient, for $\Lambda$ to spawn a visible RAAG
 subgroup.
 The `sufficient' direction would be needed to extend 
\fullref{maintheorem} to handle triangles.

\enlargethispage{1em}


\bibliographystyle{hypersshort}
\bibliography{treeRAAG}
\end{document}